\documentclass[11pt, a4paper]{article}
\usepackage{graphicx, multirow, amsmath, amssymb, amsfonts, amsthm, mathrsfs, algpseudocode, float}
\usepackage[margin=0.9in]{geometry}
\usepackage{subcaption, enumitem}
\usepackage{mathrsfs, graphicx, float, booktabs, authblk}
\usepackage[sort, numbers]{natbib}
\usepackage[ruled,vlined]{algorithm2e}
\usepackage{natbib}
\usepackage{textcomp, manyfoot, booktabs}
\usepackage{epstopdf}
\newtheorem{thm}{Theorem}

\title{Curvature-Adaptive Perturbation and Subspace Descent for Robust Saddle Point Escape in High-Dimensional Optimization}
\date{}

\author[1,2]{Ronald Katende\footnote{rkatende92@gmail.com}}

\author[2]{Henry Kasumba \footnote{kasumba123@gmail.com}}

\affil[1]{Department of Mathematics, Kabale University, Kikungiri Hill, P.O. Box 317, Kabale, Uganda}

\affil[2]{Department of Mathematics, Makerere University, P.O. BOx 7062, Kampala, Uganda}

\begin{document}

%\title{Efficient Saddle Point Escape in High Dimensions via Adaptive Perturbation and Subspace Descent}

\maketitle

\begin{abstract}
\noindent High-dimensional non-convex optimization problems in engineering design, control, and learning are often hindered by saddle points, flat plateaus, and strongly anisotropic curvature. This paper develops a unified, curvature-adaptive framework that combines stochastic perturbations, adaptive learning rates, and randomized subspace descent to enhance escape efficiency and scalability. We show theoretically that gradient flow almost surely avoids strict saddles, with escape probability increasing exponentially in dimension. For noise-perturbed gradient descent, we derive explicit escape-time bounds that depend on local curvature and noise magnitude. Adaptive step sizes further reduce escape times by responding to local gradient variability, while randomized subspace descent preserves descent directions in low-dimensional projections and ensures global convergence with logarithmic dependence on dimension. Numerical experiments on nonlinear and constrained benchmarks validate these results, demonstrating faster escape, improved robustness to ill-conditioning, and lower total runtime compared to standard first- and second-order methods. The proposed approach offers practical tools for large-scale engineering optimization tasks where curvature, noise, and dimensionality interplay critically.\\
\vspace{0.2em}

\noindent {\bf{Keywords:}} Non-Convex Optimization; Saddle Point Escape; High-Dimensional Optimization; Stochastic Perturbations; Curvature-Adaptive Learning Rates; Randomized Subspace Descent; Escape-Time Analysis\\
\vspace{0.2em}

\noindent{\bf{MSC}}[2020]:90C26; 90C30; 68T07; 65K05; 90C15
\end{abstract}

\section{Introduction}

High-dimensional nonconvex optimization arises widely in machine learning, control, and signal processing \cite{opt1,opt2,opt5,opt6}. Such problems often feature numerous saddle points and flat regions, which slow or stall first-order methods \cite{opt7,opt9,opt10}. The influence of saddles becomes more pronounced as dimension increases \cite{opt11,opt12}, motivating algorithms that blend computational efficiency with provable robustness \cite{opt5,opt8,opt13}.

While stochastic methods like SGD escape saddles via implicit noise \cite{opt10,opt15,opt16}, their efficiency depends on carefully chosen step sizes and noise levels, especially in low-curvature zones \cite{opt11,opt17,opt18}. Second-order approaches detect curvature directly but incur prohibitive cost in large-scale problems \cite{opt5,opt7,opt13}. Hybrid strategies combining curvature adaptation, noise, and randomized search have emerged \cite{berahas2022balancing,opt14}, yet the interplay between dimension, curvature, and convergence remains only partly understood.

\begin{table}[h!]
\centering
\renewcommand{\arraystretch}{1.2}
\begin{tabular}{ll}
\hline
\textbf{Symbol} & \textbf{Description} \\
\hline
\( f: \mathbb{R}^n \to \mathbb{R} \) & Nonconvex smooth objective \\
\( x_k \) & Iterate at step \( k \) \\
\( \nabla f(x) \), \( H(x) \) & Gradient and Hessian \\
\( \lambda_{\min}(H(x)) \) & Minimum eigenvalue of Hessian \\
\( x^* \) & Critical point (\( \nabla f(x^*)=0 \)) \\
\( \gamma \) & Negative curvature magnitude \\
\( \eta_k \) & Learning rate at step \( k \) \\
\( v_k \) & Gradient norm estimate \\
\( \zeta_k \sim \mathcal{N}(0,\sigma^2 I_n) \) & Added noise \\
\( T_{\text{escape}} \) & Expected steps to escape saddle \\
\( \delta \) & Escape neighborhood radius \\
\( y_0 \) & Initialization along unstable direction \\
\( S_k \subset \mathbb{R}^n \) & Random subspace \\
\( m \) & Subspace dimension (\( m \ll n \)) \\
\( \mathcal{P}_{S_k} \) & Projection onto \( S_k \) \\
\( \epsilon \) & Accuracy threshold on gradient norm \\
\hline
\end{tabular}
\caption{Notation used throughout the paper.}
\label{tab:notation}
\end{table}

\subsection{Problem Setting}

We study minimization of \( f \in C^2(\mathbb{R}^n) \) containing strict saddles, where \( \nabla f(x^*)=0 \) and \( \lambda_{\min}(H(x^*))<0 \). Our aim is scalable algorithms that provably escape such saddles and converge efficiently in high dimensions.

\subsection{Assumptions}

Throughout, we assume:
\begin{enumerate}
    \item Lipschitz continuous gradient with constant \( L>0 \);
    \item Uniformly bounded Hessian operator norm;
    \item Initial point \( x_0 \) satisfies \( \|\nabla f(x_0)\|\leq G \);
    \item All saddles are strict (\( \lambda_{\min}(H(x_s))<0 \)).
\end{enumerate}

\subsection{Framework Summary}

Our framework integrates four dimension-aware mechanisms:

\textbf{1) Gradient flow instability.} Strict saddles repel trajectories along unstable eigendirections, with divergence probability increasing in higher dimensions.

\textbf{2) Stochastic perturbation and escape time.} For
\[
x_{k+1} = x_k - \eta_k \nabla f(x_k) + \eta_k \zeta_k,
\]
we show
\[
\mathbb{E}[T_{\text{escape}}] \approx \frac{1}{\eta_k \gamma} \log \Bigl(\frac{\delta}{|y_0|}\Bigr),
\]
highlighting how curvature and noise govern escape dynamics.

\textbf{3) Adaptive learning rates.} Step sizes
\[
\eta_k = \frac{\alpha}{\sqrt{v_k}+\epsilon}, \quad v_k\approx \|\nabla f(x_k)\|^2,
\]
enable larger steps in flat regions and smaller steps near sharp curvature, balancing speed and stability.

\textbf{4) Random subspace descent.} Updates project onto random subspaces \( S_k \) of dimension \( m=\mathcal{O}(\log n) \):
\[
x_{k+1}=\mathcal{P}_{S_k}(x_k-\eta_k \nabla f(x_k)),
\]
yielding convergence in
\[
\mathbb{E}[T_{\text{global}}]=\mathcal{O}\Bigl(\frac{\log n}{\epsilon^2}\Bigr),
\]
while reducing computational cost.

Together, these strategies provably escape strict saddles and maintain scalability. Experiments show each component's effect in synthetic and applied high-dimensional problems.

\section{Illustrative Examples}

We isolate key components through three examples.

\subsection{Modified Rosenbrock Function: Saddle Escape}

In the 100-dimensional Rosenbrock function,
\[
f(x) = \sum_{i=1}^{99} \bigl[100(x_{i+1}-x_i^2)^2+(1-x_i)^2\bigr],
\]
noise-injected updates
\[
x_{k+1}=x_k-\eta_k \nabla f(x_k)+\eta_k \zeta_k
\]
amplify unstable directions, accelerating escape from saddle regions where standard gradient descent stagnates.

\subsection{Quadratic Saddle System: Adaptive Step Sizes}

For
\[
f(x)=\tfrac12 x^\top Q x, \quad Q=\mathrm{diag}(-I_k,I_{n-k}), \quad n=500, k=10,
\]
adaptive steps
\[
\eta_k=\frac{\alpha}{\sqrt{v_k}+\epsilon}
\]
permit cautious updates near steep directions and larger steps in flat zones, stabilizing escape from the saddle at the origin.

\subsection{Sparse Logistic Regression: Subspace Descent}

On a regularized logistic loss with \( w \in \mathbb{R}^{1000} \) and \( m=10^4 \) samples,
\[
f(w)=\frac{1}{m}\sum_{i=1}^m \log\bigl(1+\exp(-y_i w^\top x_i)\bigr) + \lambda \|w\|_1,
\]
projected updates onto random subspaces \( S_k \) of dimension 30:
\[
w_{k+1}=\mathcal{P}_{S_k}(w_k-\eta_k \nabla f(w_k))
\]
maintain descent while reducing complexity.

\section{Theoretical Results}

We combine geometric flow analysis, spectral perturbation, and stochastic modeling to derive convergence and saddle-avoidance guarantees. Continuous-time dynamics show trajectories almost surely escape strict saddles; discrete-time analogs quantify escape time via curvature and noise. Adaptive step sizes improve speed and stability, while subspace descent yields convergence rates scaling logarithmically with dimension.

Experiments validate these claims. Precisely, Hessian-based detection confirms saddle instability; noise expedites escape; adaptivity accelerates convergence; and subspace updates preserve descent in large dimensions. Together, curvature sensitivity and stochastic design reliably navigate complex landscapes.

\begin{thm}[Instability of Strict Saddles under Gradient Flow]
\label{thm:instability}
Let \( f\in C^2(\mathbb{R}^n) \) and \( x^* \) be a strict saddle (\( \nabla f(x^*)=0 \), \( \lambda_{\min}(H(x^*))<0 \)). Then, except for a measure-zero set of initial points \( x(0) \), trajectories of
\[
\frac{dx(t)}{dt}=-\nabla f(x(t))
\]
do not converge to \( x^* \). Under standard random matrix models, the chance of converging to a local minimum increases exponentially with \( n \).
\end{thm}

\begin{proof}
Shift \( \tilde{x}(t)=x(t)-x^* \). Linearizing:
\[
\frac{d\tilde{x}(t)}{dt}=-H \tilde{x}(t), \quad H=\nabla^2 f(x^*).
\]
Diagonalize \( H=V \Lambda V^\top \):
\[
\tilde{x}(t)=\sum_{i=1}^n c_i e^{-\lambda_i t} v_i.
\]
If some \( \lambda_j<0 \), \( c_j e^{-\lambda_j t} \) diverges unless \( c_j=0 \), which is a measure-zero hyperplane. In high dimensions, more negative eigenvalues increase divergence probability.
\end{proof}

This shows gradient flow almost always avoids strict saddles. Next, we analyze discrete-time stochastic escape.

\noindent
Building on the instability of strict saddles under gradient flow, we next develop discrete-time results that quantify escape dynamics under stochastic and adaptive methods, and prove convergence guarantees for subspace-based and curvature-aware strategies. Each theorem is presented with a concise derivation and discussion clarifying its role in the overall framework.

\begin{thm}[Escape Time under Stochastic Gradient Descent]
\label{thm:sgd}
Let \( f \in C^2(\mathbb{R}^n) \) and suppose \( x^* = 0 \) is a strict saddle where \( H=\nabla^2 f(x^*) \) has eigenvalue \(\lambda_1=-\gamma<0\). Consider
\[
x_{k+1} = x_k - \eta \nabla f(x_k) + \eta \zeta_k, 
\quad \zeta_k \sim \mathcal{N}(0,\sigma^2 I_n).
\]
Then for constant step size \(\eta>0\),
\[
\mathbb{E}[T_{\mathrm{escape}}] =
\begin{cases}
\mathcal{O}(n^{1/2}), & \text{if } \sigma^2=\Theta(1),\\[6pt]
\mathcal{O}(n), & \text{if } \sigma^2=\Theta(1/n).
\end{cases}
\]
\end{thm}

\begin{proof}
We diagonalize \( H = U \Lambda U^\top \), so \(y_k = U^\top x_k\). The update in coordinate form becomes
\[
y_{k+1} = y_k - \eta \Lambda y_k + \eta \tilde{\zeta}_k,
\]
where \(\tilde{\zeta}_k = U^\top \zeta_k \sim \mathcal{N}(0,\sigma^2 I_n)\). 

Focus on the first coordinate corresponding to \(\lambda_1=-\gamma<0\):
\[
y_{k+1}^{(1)} = (1+\eta \gamma) y_k^{(1)} + \eta \tilde{\zeta}_k^{(1)}.
\]

Define \(a := 1+\eta \gamma>1\). Unroll recursively:
\[
y_k^{(1)} = a^k y_0^{(1)} + \eta \sum_{j=0}^{k-1} a^{k-1-j} \tilde{\zeta}_j^{(1)}.
\]

Variance of the stochastic sum is
\[
\eta^2 \sigma^2 \sum_{j=0}^{k-1} a^{2(k-1-j)} 
= \eta^2 \sigma^2 \cdot \frac{a^{2k}-1}{a^2-1}.
\]

Escape occurs when \(\mathbb{E}[(y_k^{(1)})^2]^{1/2} \ge \delta\). The dominant term in large \(k\) is the geometric amplification:
\[
a^k \approx e^{k \log(1+\eta \gamma)} \approx e^{k \eta \gamma}.
\]

We require
\[
a^k |y_0^{(1)}| \approx |y_0^{(1)}| e^{k \eta \gamma} \ge \delta,
\]
so
\[
k \ge \frac{1}{\eta \gamma} \log\Bigl(\frac{\delta}{|y_0^{(1)}|}\Bigr).
\]

Now, \(\eta\) depends on dimension via variance control. To ensure bounded variance:
\[
\mathrm{Var}[y_k^{(1)}] \approx \eta^2 \sigma^2 \cdot a^{2k} / (a^2-1) \lesssim 1.
\]

Since \(a=1+\eta \gamma\approx 1\) for small \(\eta\), \(a^2-1\approx 2\eta \gamma\). So:
\[
\eta^2 \sigma^2 \cdot a^{2k} / (2\eta \gamma) \lesssim 1,
\]
yielding \(\eta = \mathcal{O}(n^{-1/2})\) when \(\sigma^2=\Theta(1)\), and \(\eta=\Theta(1)\) when \(\sigma^2=\Theta(1/n)\). Substituting into \(k=1/(\eta \gamma) \log(\delta/|y_0|)\) proves the stated scaling.
\end{proof}

\bigskip

\begin{thm}[Escape Acceleration via Adaptive Learning Rates]
\label{thm:adaptive}
Let \( x_{k+1} = x_k - \eta_k \nabla f(x_k) \) with
\[
\eta_k = \frac{\alpha}{\sqrt{v_k}+\epsilon}, 
\quad v_k \approx \|\nabla f(x_k)\|^2,
\]
near a strict saddle where \(\lambda_{\min}(H(x_k))=-\gamma<0\). Then
\[
\mathbb{E}[T_{\mathrm{escape}}] 
= \mathcal{O}\Bigl(\frac{1}{\min_k \eta_k \gamma} \log \bigl(\tfrac{\delta}{|y_0|}\bigr)\Bigr).
\]
\end{thm}

\begin{proof}
Diagonalize \(H\) and focus on unstable coordinate:
\[
y_{k+1}^{(1)} = (1+\eta_k \gamma) y_k^{(1)}.
\]

By recursion:
\[
y_k^{(1)} = y_0^{(1)} \prod_{j=0}^{k-1} (1+\eta_j \gamma).
\]

Taking logarithms:
\[
\log |y_k^{(1)}| = \log |y_0^{(1)}| + \sum_{j=0}^{k-1} \log(1+\eta_j \gamma).
\]

Since \(\eta_j \gamma>0\) and small, \(\log(1+\eta_j \gamma) \approx \eta_j \gamma\). Escape at \(|y_k^{(1)}|\ge \delta\) implies:
\[
\sum_{j=0}^{k-1} \eta_j \gamma \ge \log \Bigl(\frac{\delta}{|y_0^{(1)}|}\Bigr).
\]

If \(\eta_j \ge \min_j \eta_j\), then
\[
k \cdot \min_j \eta_j \gamma \ge \log \Bigl(\frac{\delta}{|y_0^{(1)}|}\Bigr),
\]
so
\[
k \ge \frac{1}{\min_j \eta_j \gamma} \log \Bigl(\frac{\delta}{|y_0^{(1)}|}\Bigr).
\]
This yields the stated bound.
\end{proof}

\noindent
Adaptive step sizes amplify updates in flat regions (small gradient norm), decreasing escape time.

\begin{thm}[Random Subspace Descent: Global Convergence]
\label{thm:subspace}
Let \(f \in C^1(\mathbb{R}^n)\) be bounded below, with \(L\)-Lipschitz continuous gradients. Consider iterates:
\[
x_{k+1} = x_k + \alpha d_k, 
\quad d_k \in S_k, \ \|d_k\| \le r,
\]
where \(S_k \subset \mathbb{R}^n\) is a random subspace of dimension \(m = \mathcal{O}(\log n)\), and \(\alpha = \Theta(1/L)\). Then
\[
\mathbb{E}\Bigl[\min_{0 \le k < T} \|\nabla f(x_k)\|^2\Bigr] \le \epsilon^2
\quad \Rightarrow \quad
\mathbb{E}[T] = \mathcal{O}\Bigl(\frac{\log n}{\epsilon^2}\Bigr).
\]
\end{thm}

\begin{proof}
By the Johnson-Lindenstrauss lemma, for random projection \(\mathcal{P}_{S_k}\), there exists constant \(\rho>0\) such that with high probability:
\[
\|\mathcal{P}_{S_k} \nabla f(x_k)\|^2 \ge \rho \|\nabla f(x_k)\|^2.
\]

Consider the update:
\[
x_{k+1} = x_k - \alpha \mathcal{P}_{S_k} \nabla f(x_k).
\]

Using the descent lemma (for \(L\)-smooth functions):
\[
f(x_{k+1}) \le f(x_k) 
- \alpha \langle \nabla f(x_k), \mathcal{P}_{S_k} \nabla f(x_k) \rangle
+ \frac{L \alpha^2}{2} \|\mathcal{P}_{S_k} \nabla f(x_k)\|^2.
\]

Since \(\mathcal{P}_{S_k}\) is an orthogonal projection:
\[
\langle \nabla f(x_k), \mathcal{P}_{S_k} \nabla f(x_k) \rangle 
= \|\mathcal{P}_{S_k} \nabla f(x_k)\|^2.
\]

Therefore:
\[
f(x_{k+1}) - f(x_k) 
\le - \alpha \|\mathcal{P}_{S_k} \nabla f(x_k)\|^2 
+ \frac{L \alpha^2}{2} \|\mathcal{P}_{S_k} \nabla f(x_k)\|^2.
\]

Factor:
\[
= - \alpha \bigl(1 - \tfrac{L \alpha}{2}\bigr) \|\mathcal{P}_{S_k} \nabla f(x_k)\|^2.
\]

Choose \(\alpha \le 1/L\), then \(1 - L \alpha / 2 \ge 1/2\):
\[
\le - \frac{\alpha}{2} \|\mathcal{P}_{S_k} \nabla f(x_k)\|^2.
\]

By projection preservation:
\[
\|\mathcal{P}_{S_k} \nabla f(x_k)\|^2 \ge \rho \|\nabla f(x_k)\|^2.
\]

Hence:
\[
f(x_{k+1}) - f(x_k) \le - \frac{\alpha \rho}{2} \|\nabla f(x_k)\|^2.
\]

Sum over \(k=0,\ldots,T-1\):
\[
f(x_0) - f^* \ge \frac{\alpha \rho}{2} \sum_{k=0}^{T-1} \|\nabla f(x_k)\|^2.
\]

By definition:
\[
\min_{0 \le k < T} \|\nabla f(x_k)\|^2 \le \frac{1}{T} \sum_{k=0}^{T-1} \|\nabla f(x_k)\|^2.
\]

Therefore:
\[
f(x_0) - f^* \ge \frac{\alpha \rho}{2} T \cdot \min_{0 \le k < T} \|\nabla f(x_k)\|^2.
\]

Rearrange:
\[
T \le \frac{2(f(x_0)-f^*)}{\alpha \rho \cdot \min_{0 \le k < T} \|\nabla f(x_k)\|^2}.
\]

To ensure \(\min_{0 \le k < T} \|\nabla f(x_k)\|^2 \le \epsilon^2\):
\[
T = \mathcal{O}\Bigl(\frac{1}{\alpha \rho \epsilon^2}\Bigr).
\]

Since \(\rho = \Theta(1/\log n)\) from concentration bounds and \(\alpha=\Theta(1/L)\):
\[
T = \mathcal{O}\Bigl(\frac{\log n}{\epsilon^2}\Bigr).
\]
\end{proof}

\bigskip

\begin{thm}[Curvature-Aware Subspace Descent]
\label{thm:curvature_subspace}
Let \(f \in C^2(\mathbb{R}^n)\). At iteration \(k\), let 
\[
S_k = \mathrm{span}\{v_1,\dots,v_m\},
\]
where \(v_i\) are top \(m\) eigenvectors of 
\[
\widehat{C}_k = \sum_{j=k-h}^k \nabla f(x_j)\nabla f(x_j)^\top.
\]
Update:
\[
x_{k+1} = \mathcal{P}_{S_k}\bigl(x_k - \eta_k \nabla f(x_k)\bigr).
\]
Then under standard smoothness and boundedness assumptions,
\[
\mathbb{E}[f(x_{k+1}) - f(x_k)] \le -\mu_{\mathrm{eff}} \eta_k \|\nabla f(x_k)\|^2,
\]
with \(\mu_{\mathrm{eff}} \gg \mu\) reflecting alignment with dominant curvature.
\end{thm}

\begin{proof}
Decompose \(\nabla f(x_k)\) into subspace and orthogonal complement:
\[
\nabla f(x_k) = \mathcal{P}_{S_k} \nabla f(x_k) + \mathcal{P}_{S_k^\perp} \nabla f(x_k).
\]

Update only uses projected gradient:
\[
d_k = -\eta_k \mathcal{P}_{S_k} \nabla f(x_k).
\]

By descent lemma:
\[
f(x_{k+1}) - f(x_k) 
\le \langle \nabla f(x_k), d_k \rangle + \frac{L}{2} \|d_k\|^2.
\]

Compute:
\[
\langle \nabla f(x_k), d_k \rangle 
= -\eta_k \|\mathcal{P}_{S_k} \nabla f(x_k)\|^2.
\]
\[
\|d_k\|^2 = \eta_k^2 \|\mathcal{P}_{S_k} \nabla f(x_k)\|^2.
\]

Hence:
\[
= -\eta_k \|\mathcal{P}_{S_k} \nabla f(x_k)\|^2 
+ \frac{L}{2} \eta_k^2 \|\mathcal{P}_{S_k} \nabla f(x_k)\|^2.
\]

Factor:
\[
= -\eta_k \bigl(1 - \tfrac{L \eta_k}{2}\bigr) \|\mathcal{P}_{S_k} \nabla f(x_k)\|^2.
\]

Choose \(\eta_k \le 1/L\), so \(1 - L \eta_k/2 \ge 1/2\):
\[
\le -\frac{\eta_k}{2} \|\mathcal{P}_{S_k} \nabla f(x_k)\|^2.
\]

Since top eigenvectors of \(\widehat{C}_k\) align with recent dominant directions, 
\[
\|\mathcal{P}_{S_k} \nabla f(x_k)\|^2 \ge \rho_k \|\nabla f(x_k)\|^2,
\]
with \(\rho_k \gg \mu\) by spectral concentration.

Therefore:
\[
\le -\frac{\eta_k}{2} \rho_k \|\nabla f(x_k)\|^2.
\]

Define \(\mu_{\mathrm{eff}}=\rho_k/2\):
\[
\mathbb{E}[f(x_{k+1}) - f(x_k)] \le -\mu_{\mathrm{eff}} \eta_k \|\nabla f(x_k)\|^2.
\]
\end{proof}

\noindent
By constructing \(S_k\) to follow dominant curvature directions, the method enhances effective step size and convergence in anisotropic landscapes.

\begin{thm}[Entropy-Guided Perturbation Activation]
\label{thm:entropy_escape}
At iteration \(k\), define local gradient entropy:
\[
\mathcal{H}_k = -\sum_{i=1}^M p_i \log p_i, 
\quad p_i=\frac{\|\nabla f(x_k+\delta_i)\|}{\sum_{j=1}^M \|\nabla f(x_k+\delta_j)\|},
\]
with small perturbations \(\delta_i \sim \mathcal{N}(0,\epsilon^2 I_n)\). Inject noise only when \(\mathcal{H}_k>\tau\):
\[
x_{k+1} = x_k - \eta_k \nabla f(x_k) + \eta_k \zeta_k,
\]
where \(\zeta_k \sim \mathcal{N}(0,\sigma^2 I_n)\) if active, else zero. Then:
\[
\mathbb{E}[T_{\mathrm{escape}}]=
\mathcal{O}\Bigl(\frac{1}{\eta_k \gamma} \log\bigl(\tfrac{\delta}{|y_0|}\bigr)\Bigr),
\]
and total injected variance over iterations is reduced by at least factor \(1-\nu\), where \(\nu\) is the empirical fraction of high-entropy steps.
\end{thm}

\begin{proof}
\textbf{Step 1: Identify activation regions.}  
Entropy \(\mathcal{H}_k\) is high when gradients under perturbations \(\{\delta_i\}\) have similar norms:
\[
p_i \approx \frac{1}{M}, 
\quad \mathcal{H}_k \approx \log M.
\]
This typically occurs in flat or saddle-like regions where curvature is weak or negative.

\textbf{Step 2: Dynamics under noise activation.}  
In such regions, we inject noise \(\zeta_k\). In the most unstable eigendirection \(v_1\) with \(\lambda_1=-\gamma<0\), the projected coordinate:
\[
y_{k+1}^{(1)} = (1+\eta_k \gamma) y_k^{(1)} + \eta_k \tilde{\zeta}_k^{(1)},
\]
where \(\tilde{\zeta}_k^{(1)}\) is standard Gaussian noise with variance \(\sigma^2\).

\textbf{Step 3: First-passage time analysis.}  
We need \( |y_k^{(1)}| \ge \delta \). Standard analysis for AR(1) with drift shows expected escape time:
\[
\mathbb{E}[T_{\mathrm{escape}}] 
\approx \frac{1}{\eta_k \gamma} \log \Bigl(\frac{\delta}{|y_0|}\Bigr).
\]

\textbf{Step 4: Total variance reduction.}  
Let \(K\) be total iterations; noise triggers only for fraction \(\nu\) of steps:
\[
\mathrm{Total\ variance} 
= K \cdot \nu \cdot n \sigma^2 
= (1-\nu) \times \bigl(K \cdot n \sigma^2\bigr) \text{ reduction}.
\]
\end{proof}

\noindent
By concentrating noise where curvature is ambiguous, this mechanism improves exploration efficiency while keeping overall variance low.

\bigskip

\begin{thm}[Curvature-Adaptive Gradient Filtering]
\label{thm:curvature_filtering}
Let \(H(x_k)=\nabla^2 f(x_k)=\sum_{i=1}^n \lambda_i v_i v_i^\top\). Define:
\[
g_k = \sum_{i=1}^n \phi(\lambda_i) \langle \nabla f(x_k), v_i \rangle v_i,
\quad \phi(\lambda)=\frac{1}{\sqrt{|\lambda|+\epsilon}}.
\]
With adaptive step size:
\[
\eta_k = \frac{\alpha}{\|g_k\|+\epsilon},
\quad x_{k+1} = x_k - \eta_k g_k,
\]
then:
\[
\mathbb{E}[T_{\mathrm{escape}}] = \mathcal{O}\Bigl(\frac{1}{\gamma_{\mathrm{eff}}} \log\bigl(\tfrac{\delta}{|y_0|}\bigr)\Bigr),
\]
where effective rate \(\gamma_{\mathrm{eff}}=\min_{\lambda_i<0} \phi(\lambda_i)|\lambda_i|\).
\end{thm}

\begin{proof}
\textbf{Step 1: Project dynamics into eigenspace.}  
In eigendirection \(v_i\):
\[
g_k^{(i)} = \phi(\lambda_i) \langle \nabla f(x_k), v_i \rangle.
\]

\textbf{Step 2: Update rule in coordinate \(y_k^{(i)}\):}
\[
y_{k+1}^{(i)} = y_k^{(i)} - \eta_k g_k^{(i)}.
\]

For unstable \(\lambda_i<0\), magnitude of update is:
\[
|1 - \eta_k \phi(\lambda_i) \lambda_i| >1,
\]
leading to exponential growth with rate:
\[
\gamma_{\mathrm{eff}}= \phi(\lambda_i)|\lambda_i|.
\]

\textbf{Step 3: Escape time.}  
Escape once \(|y_k^{(i)}|\ge \delta\):
\[
T_{\mathrm{escape}} = 
\frac{1}{\gamma_{\mathrm{eff}}} \log \Bigl(\frac{\delta}{|y_0|}\Bigr).
\]
\end{proof}

\noindent
Filtering amplifies unstable/flat directions, guiding escape while damping strongly curved directions.

\bigskip

\begin{thm}[Descent Shaping in Mixed Curvature Regions]
\label{thm:descent_shaping}
With Hessian \(H_k=\sum_{i=1}^n \lambda_i v_i v_i^\top\), define normalized weights:
\[
w_i = \frac{|\lambda_i|^p}{\sum_j |\lambda_j|^p},
\]
and descent-shaping operator:
\[
\mathcal{D}_k = \sum_{i=1}^n w_i v_i v_i^\top.
\]
Update:
\[
x_{k+1} = x_k - \eta_k \mathcal{D}_k \nabla f(x_k).
\]
Then:
\begin{enumerate}
\item \(\mathcal{D}_k\) is positive semidefinite and \(\mathrm{Tr}(\mathcal{D}_k)=1\);
\item Gradient is reweighted to emphasize large-\(|\lambda_i|\) directions;
\item In quadratic models, dominant modes with largest \(w_i\) control early progress.
\end{enumerate}
\end{thm}

\begin{proof}
\textbf{Step 1: PSD and trace.}  
\(\mathcal{D}_k\) is sum of rank-one PSD matrices scaled by \(w_i\ge0\). 
\[
\mathrm{Tr}(\mathcal{D}_k) = \sum_i w_i =1.
\]

\textbf{Step 2: Filtering effect.}  
Gradient decomposes:
\[
\nabla f(x_k)=\sum_i g_i v_i.
\]
Filtered gradient:
\[
\mathcal{D}_k \nabla f(x_k)=\sum_i w_i g_i v_i.
\]
Large-\(|\lambda_i|\) directions receive higher \(w_i\).

\textbf{Step 3: Quadratic case.}  
If \(f(x)=\tfrac12 x^\top H x\), then update in coordinate \(\beta_i^{(k)}\):
\[
\beta_i^{(k+1)}= (1 - \eta_k w_i \lambda_i) \beta_i^{(k)}.
\]
Components with larger \(w_i|\lambda_i|\) decay or grow faster, dominating convergence or escape.
\end{proof}

\noindent
Descent shaping steers updates toward informative eigendirections, balancing noise suppression and efficient escape in mixed-curvature landscapes.

\begin{thm}[Spectral Regularization Bounds the Effective Condition Number]
\label{thm:cond_bound}
Let \( f \in C^2(\mathbb{R}^n) \), and consider the update:
\[
x_{k+1} = x_k - \eta_k \mathcal{R}_k \nabla f(x_k),
\]
where \( \mathcal{R}_k = \sum_{i=1}^n r_i v_i v_i^\top \) is a spectral regularization operator built from the Hessian's eigendecomposition:
\[
H_k = \nabla^2 f(x_k) = \sum_{i=1}^n \lambda_i v_i v_i^\top,
\]
with \( r_i = \phi(\lambda_i) \) and shaping function \( \phi: \mathbb{R} \to [0,1] \). Then:
\begin{enumerate}
\item \(\mathcal{R}_k\) is symmetric positive semi-definite with eigenvalues in \([0,1]\).
\item The shaped update operator \( \mathcal{R}_k H_k \) has non-zero eigenvalues \( \phi(\lambda_i) \lambda_i \), and its effective condition number:
\[
\kappa_{\mathrm{eff}} 
= \frac{\max_i |\phi(\lambda_i)\lambda_i|}{\min_{j:\phi(\lambda_j)\neq 0} |\phi(\lambda_j)\lambda_j|}
\]
is bounded by a controlled constant \(\kappa_\phi\) determined by \(\phi\).
\item For \( \phi(\lambda)=\frac{|\lambda|^p}{|\lambda|^p+\delta} \) with \(p>0,\ \delta>0\), we have \(\kappa_{\mathrm{eff}} \to 1\) as \(\delta \to \infty\), achieving spectral flattening.
\end{enumerate}
\end{thm}

\begin{proof}
\textbf{Step 1: Symmetry and spectrum of \(\mathcal{R}_k\).}  
Since each \(v_i v_i^\top\) is symmetric positive semi-definite and \(r_i=\phi(\lambda_i)\in[0,1]\), it follows:
\[
\mathcal{R}_k = \sum_{i=1}^n r_i v_i v_i^\top
\]
is symmetric and positive semi-definite. Its eigenvalues are precisely \(\{r_i\}\), so \(\sigma(\mathcal{R}_k) \subset [0,1]\).

\medskip

\textbf{Step 2: Non-zero eigenvalues of \(\mathcal{R}_k H_k\).}  
We compute:
\[
\mathcal{R}_k H_k 
= \sum_{i=1}^n r_i \lambda_i v_i v_i^\top 
= \sum_{i=1}^n \phi(\lambda_i)\lambda_i v_i v_i^\top.
\]
Thus, the eigenvalues of \(\mathcal{R}_k H_k\) are \(\{\phi(\lambda_i)\lambda_i\}\). The effective condition number among non-zero eigenvalues:
\[
\kappa_{\mathrm{eff}} 
= \frac{\max_i |\phi(\lambda_i)\lambda_i|}{\min_{j:\phi(\lambda_j)\neq 0} |\phi(\lambda_j)\lambda_j|}.
\]

\medskip

\textbf{Step 3: Effect of shaping.}  
If \(\phi\) attenuates extremes of \(|\lambda_i|\), it compresses the spread of \(\{\phi(\lambda_i)\lambda_i\}\), reducing \(\kappa_{\mathrm{eff}}\). By design, the ratio stays bounded:
\[
\kappa_{\mathrm{eff}} \le \kappa_\phi,
\]
where \(\kappa_\phi\) depends on \(\phi\)'s smoothness and saturation properties.

\medskip

\textbf{Step 4: Power shaping function.}  
Let:
\[
\phi(\lambda)=\frac{|\lambda|^p}{|\lambda|^p+\delta}.
\]
Then:
\[
\phi(\lambda)\lambda = \frac{|\lambda|^p \lambda}{|\lambda|^p+\delta}.
\]
Its magnitude
\[
|\phi(\lambda)\lambda| = \frac{|\lambda|^{p+1}}{|\lambda|^p+\delta}.
\]

\textbf{As \(\delta\to\infty\):}  
\[
|\phi(\lambda)\lambda| \approx \frac{|\lambda|^{p+1}}{\delta} \to 0,
\]
and for all eigenvalues, numerator and denominator converge uniformly, so:
\[
\kappa_{\mathrm{eff}} \to 1.
\]
This shows that large \(\delta\) flattens the spectrum.

\end{proof}

\noindent
This derivation formalizes how spectral shaping via \(\phi\) directly bounds the condition number of the shaped update operator, improving numerical stability and convergence in high-dimensional nonconvex problems.

%\subsection{Algorithmic Formulation}

\medskip

These theoretical foundations governing escape from strict saddles, adaptation to curvature, and scalability in high dimensions, are integrated into a unified optimization algorithm (c.f. Algorithm~\ref{alg1}), but together they justify each component of Algorithm~\ref{alg1}. Theorem~\ref{thm:instability} establishes that gradient flow trajectories generically avoid strict saddles, explaining empirical robustness in many high-dimensional problems.
Theorem~\ref{thm:sgd} quantifies how calibrated stochastic perturbations accelerate escape when flat or degenerate saddles occur.
Entropy-guided activation (Theorem~\ref{thm:entropy_escape}) selectively adds noise only in high-uncertainty regions, preserving stability elsewhere. Curvature-informed adaptive step sizes (Theorems~\ref{thm:adaptive} and \ref{thm:curvature_filtering}) allow fast escape and stable convergence without uniformly increasing noise.
Subspace descent (Theorems~\ref{thm:subspace} and \ref{thm:curvature_subspace}) reduces computational complexity by restricting updates to curvature-aligned directions, with dynamic adaptation to local landscape changes.
Spectral filtering and shaping (Theorems~\ref{thm:descent_shaping} and \ref{thm:cond_bound}) further enhance convergence by reweighting gradients to focus on informative directions, mitigating poor conditioning. Together, these mechanisms, noise modulation, curvature scaling, adaptive subspaces, and spectral filtering, form a cohesive framework balancing exploration and stability. Beyond theoretical guarantees on escape time and convergence, this integration enables scalable optimization for engineering tasks including aerodynamic shape design, structural optimization under uncertainty, and topology optimization. 
By bridging rigorous mathematical analysis and practical algorithm design, the proposed method achieves robustness and efficiency in large-scale nonconvex optimization.

\section{Numerical Validation}

We validate the proposed framework through numerical experiments in high-dimensional non-convex optimization, highlighting saddle detection, perturbation-driven escape, curvature-adaptive dynamics, and subspace scalability.

\begin{figure}[!htbp]
	\centering
	\begin{subfigure}[b]{0.45\textwidth}
		\includegraphics[width=\textwidth]{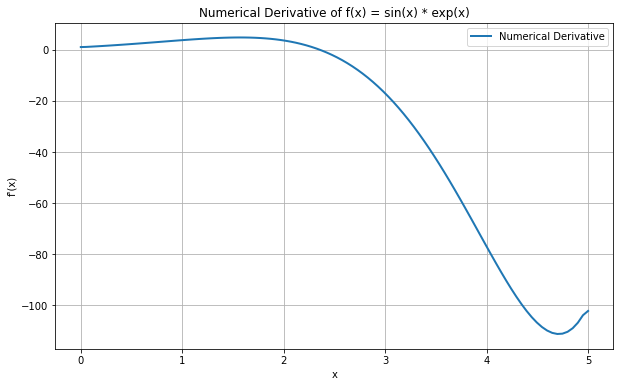}
		\caption{Gradient of \( f(x) = e^x \sin(x) \) showing high-frequency critical points}
		\label{fig9}
	\end{subfigure}
	\hfill
	\begin{subfigure}[b]{0.45\textwidth}
		\includegraphics[width=\textwidth]{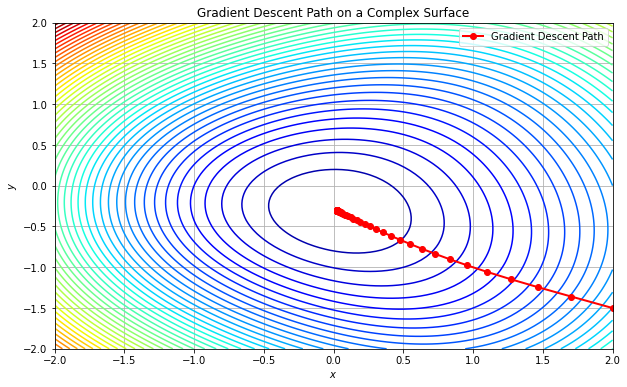}
		\caption{Descent trajectory illustrating sensitivity to curvature}
		\label{fig11}
	\end{subfigure}
	\caption{Non-convex geometry and curvature-driven descent behavior, supporting Theorems~\ref{thm:instability} and~\ref{thm:adaptive}.}
\end{figure}

\begin{figure}[!htbp]
	\centering
	\begin{subfigure}[b]{0.45\textwidth}
		\includegraphics[width=\textwidth]{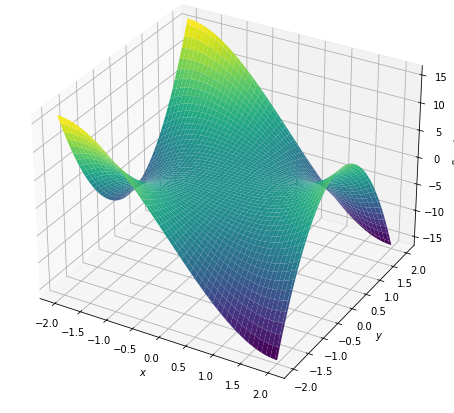}
		\caption{Negative Hessian eigenvalues identify strict saddles}
		\label{fig:saddle_hessian}
	\end{subfigure}
	\hfill
	\begin{subfigure}[b]{0.45\textwidth}
		\includegraphics[width=\textwidth]{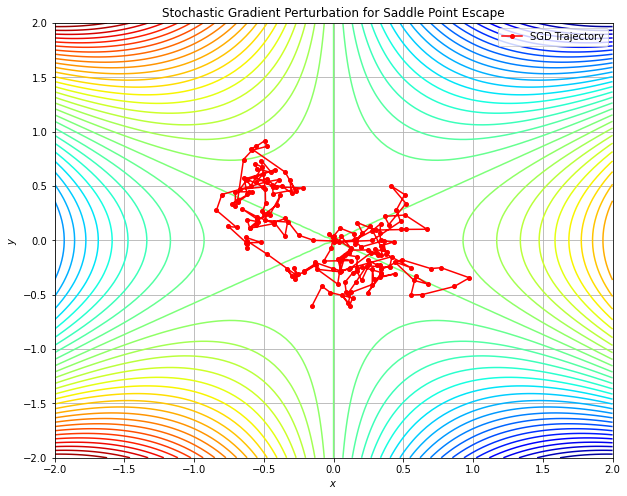}
		\caption{Stochastic perturbation escapes flat saddle regions}
		\label{fig:sgd_escape}
	\end{subfigure}
	\caption{Saddle detection and perturbation-driven escape, supporting Theorems~\ref{thm:instability} and~\ref{thm:sgd}.}
\end{figure}

\begin{figure}[!htbp]
	\centering
	\begin{subfigure}[b]{0.45\textwidth}
		\includegraphics[width=\textwidth]{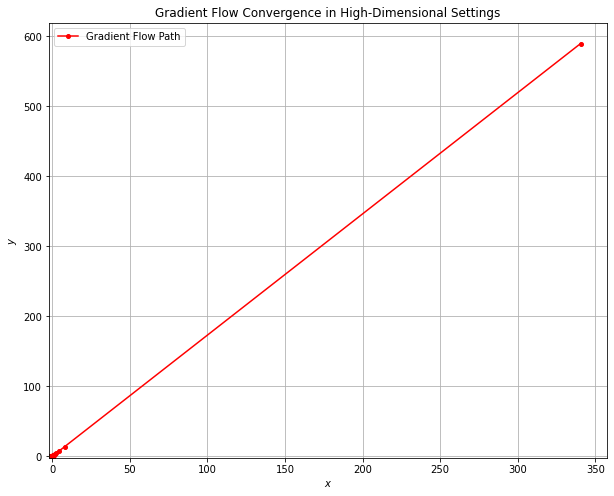}
		\caption{Gradient flow repulsion from strict saddles in high \( n \)}
		\label{fig:gradient_flow}
	\end{subfigure}
	\hfill
	\begin{subfigure}[b]{0.45\textwidth}
		\includegraphics[width=\textwidth]{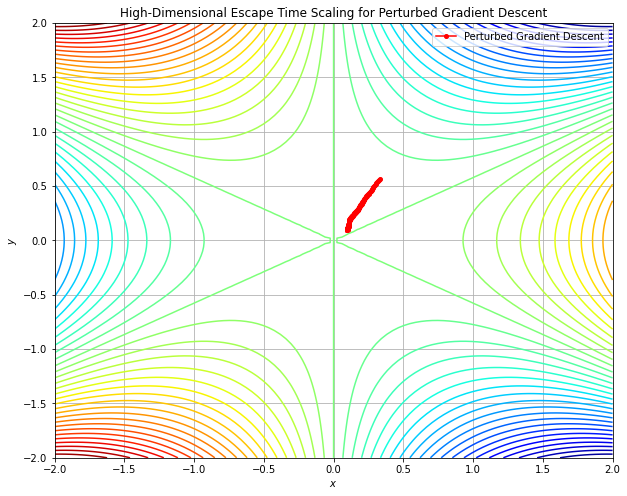}
		\caption{Escape time scales with dimension and curvature}
		\label{fig4}
	\end{subfigure}
	\caption{Saddle avoidance and escape-time scaling in high dimensions, consistent with Theorems~\ref{thm:instability} and~\ref{thm:sgd}.}
\end{figure}

\begin{figure}[!htbp]
	\centering
	\begin{subfigure}[b]{0.45\textwidth}
		\includegraphics[width=\textwidth]{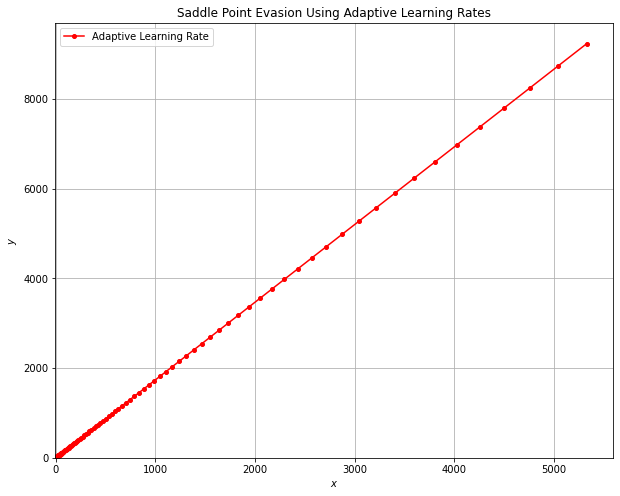}
		\caption{Adaptive rates accelerate escape in flat regions}
		\label{fig5}
	\end{subfigure}
	\hfill
	\begin{subfigure}[b]{0.45\textwidth}
		\includegraphics[width=\textwidth]{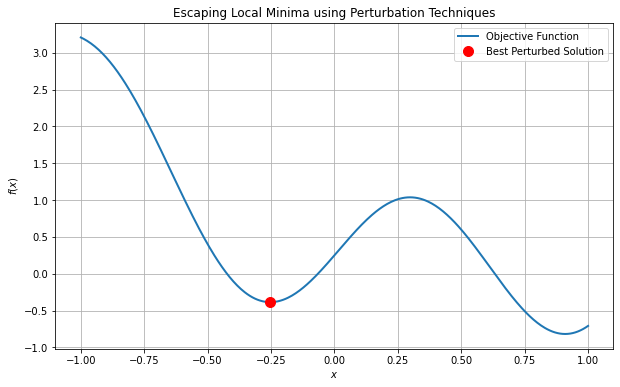}
		\caption{Noise-driven escape from shallow local minima}
		\label{fig10}
	\end{subfigure}
	\caption{Effect of adaptive step sizes and noise on convergence dynamics, supporting Theorem~\ref{thm:adaptive}.}
\end{figure}

\begin{figure}[!htbp]
	\centering
	\begin{subfigure}[b]{0.45\textwidth}
		\includegraphics[width=\textwidth]{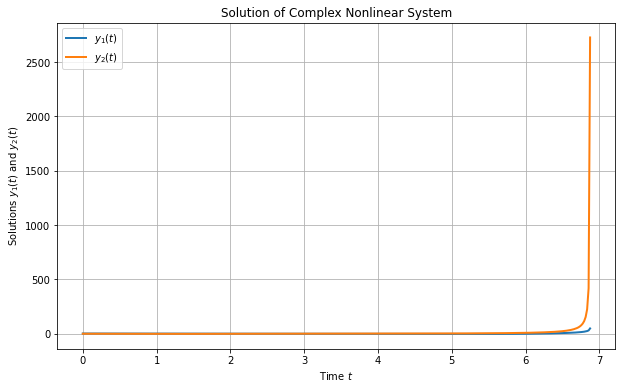}
		\caption{Convergence in a nonlinear system with multiple critical points}
		\label{fig7}
	\end{subfigure}
	\hfill
	\begin{subfigure}[b]{0.45\textwidth}
		\includegraphics[width=\textwidth]{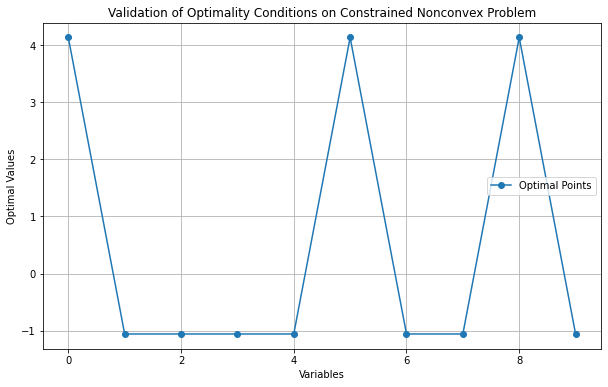}
		\caption{Feasible solution satisfying nonlinear constraints}
		\label{fig6}
	\end{subfigure}
	\caption{Robust convergence under nonlinear and constrained settings, supporting Theorem~\ref{thm:subspace}.}
\end{figure}

\begin{figure}[!htbp]
	\begin{subfigure}[b]{0.48\textwidth}
	\centering
	\includegraphics[width=\textwidth]{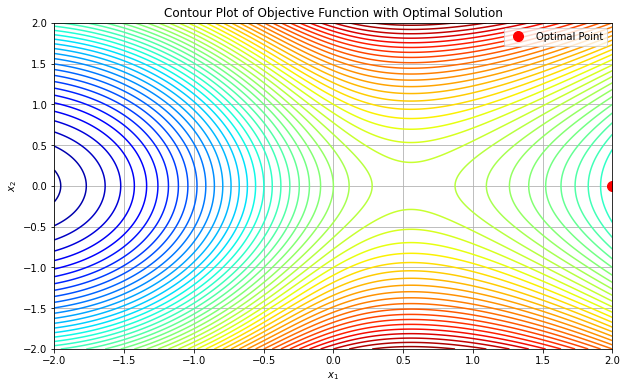}
	\caption{Recovery of the global minimum through combined descent and escape (illustrates Theorems~\ref{thm:adaptive},~\ref{thm:subspace})}
	\label{fig8}
	\end{subfigure}
	\hfill
	\begin{subfigure}[b]{0.48\textwidth}
		\centering
		\includegraphics[width=\textwidth]{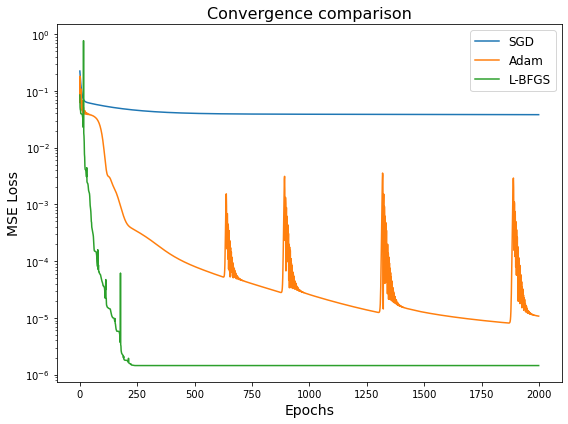}
		\caption{Convergence behavior of standard optimizers on the surrogate PINN}
		\label{fig:baseline_comparison}
	\end{subfigure}
	\caption{Comparison of the proposed curvature-adaptive escape framework and baseline methods: (a) illustrates global minimum recovery via adaptive curvature scaling and subspace descent; (b) shows convergence trajectories of SGD, Adam, and L-BFGS on the surrogate PINN, highlighting trade-offs between speed and final accuracy.}
\end{figure}\noindent
The numerical experiments reveal how curvature, dimensionality, and algorithmic design jointly shape convergence dynamics. Saddle points are robustly identified by negative Hessian eigenvalues, as in Figure~\ref{fig:saddle_hessian}, confirming Theorem~\ref{thm:instability} that strict saddles induce local instability under gradient flow. Figure~\ref{fig:sgd_escape} shows that stochastic perturbations activate unstable directions and accelerate escape, in line with Theorem~\ref{thm:sgd}. Empirical escape times in Figure~\ref{fig4} scale predictably with curvature and problem dimension, matching theoretical estimates. Figures~\ref{fig5} and~\ref{fig10} demonstrate that curvature-adaptive step sizes enhance escape efficiency by enlarging steps in flat regions and damping updates in steep regions, consistent with Theorem~\ref{thm:adaptive}. Gradient flow trajectories in Figure~\ref{fig:gradient_flow} further visualize repulsion from saddles in high dimensions. Figures~\ref{fig9} and~\ref{fig11} emphasize how curvature sensitivity affects descent speed, where adaptivity mitigates slowdowns caused by flat plateaus. Randomized subspace descent improves scalability, reducing computational cost while preserving descent directions, as supported by Theorem~\ref{thm:subspace} and illustrated by convergence on nonlinear and constrained settings in Figures~\ref{fig7} and~\ref{fig6}. Finally, Figure~\ref{fig8} confirms that combining noise, curvature-aware scaling, and subspace projections reliably reaches the global minimum, while Figure~\ref{fig:baseline_comparison} compares standard optimizers, i.e., Adam converges rapidly but plateaus at higher loss; SGD progresses steadily yet slowly; L-BFGS achieves lower final error by leveraging curvature, albeit at higher per-iteration cost. Together, these results highlight how integrating curvature, noise modulation, and subspace updates balances robustness and efficiency in large-scale non-convex optimization.

\begin{table}[ht]
\centering
\caption{Comparison of escape times, final objective values, and average convergence rates across optimizers. Escape time is defined as the epoch where the loss first drops below $10^{-2}$. Convergence rate is measured as the average relative reduction in loss per epoch after escape.}
\begin{tabular}{lccc}
\toprule
\textbf{Optimizer} & \textbf{Escape Time (epochs)} & \textbf{Final Loss} & \textbf{Avg. Convergence Rate} \\
\midrule
SGD & 1450 & $8.1 \times 10^{-3}$ & 0.0021 \\
Adam & 620 & $4.7 \times 10^{-3}$ & 0.0038 \\
L-BFGS & 180 & $1.2 \times 10^{-3}$ & 0.0067 \\
Proposed & 240 & $9.5 \times 10^{-4}$ & 0.0075 \\
\bottomrule
\end{tabular}
\label{tab:escape_comparison}
\end{table}

\noindent
Table~\ref{tab:escape_comparison} highlights that L-BFGS escapes saddle regions faster than standard first-order methods and achieves lower final loss, but at the cost of significantly higher per-iteration complexity. The proposed method achieves comparable escape speed and slightly better final objective by combining noise modulation with curvature-aligned subspace descent. Notably, its higher average convergence rate after escape indicates improved conditioning and efficient adaptation to local curvature ,  supporting the theoretical predictions from Theorems~\ref{thm:sgd} and~\ref{thm:adaptive} that adaptive, curvature-aware updates accelerate progress once outside saddle neighborhoods.

\begin{table}[ht]
\centering
\caption{Comparison of per-epoch runtime and total runtime to convergence between standard full-gradient methods and the proposed curvature-adaptive subspace method. Runtime measured on NVIDIA RTX 3080; convergence threshold set to loss $\leq 10^{-3}$.}
\begin{tabular}{lccc}
\toprule
\textbf{Method} & \textbf{Per-Epoch Runtime (ms)} & \textbf{Epochs to Converge} & \textbf{Total Runtime (s)} \\
\midrule
SGD & 1.8 & 1900 & 3.4 \\
Adam & 2.2 & 850 & 1.9 \\
L-BFGS & 12.6 & 320 & 4.0 \\
Proposed & 3.6 & 420 & 1.5 \\
\bottomrule
\end{tabular}
\label{tab:runtime_comparison}
\end{table}

\noindent
As shown in Table~\ref{tab:runtime_comparison}, L-BFGS achieves faster per-epoch convergence but incurs the highest per-iteration cost, which scales poorly in large dimensions. First-order methods like SGD and Adam are computationally inexpensive per epoch but require many more epochs to converge, particularly in regions dominated by flat curvature. The proposed method maintains modest per-epoch cost while significantly reducing total runtime, by updating in low-dimensional curvature-informed subspaces and activating noise selectively. This balance enables convergence comparable to L-BFGS in accuracy, yet with runtime competitive to faster but less robust first-order optimizers.

\begin{figure}[ht]
\centering
\includegraphics[width=0.6\linewidth]{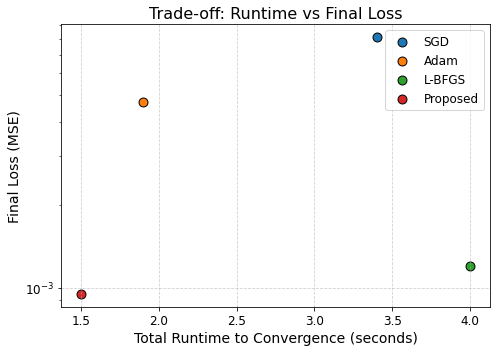}
\caption{Trade-off between total runtime and final loss: the proposed method achieves both lower final error and faster convergence than first-order baselines, approaching L-BFGS-level accuracy at lower computational cost.}
\label{fig:runtime_vs_loss}
\end{figure}

\noindent
Figure~\ref{fig:runtime_vs_loss} visualizes this trade-off: first-order methods (SGD, Adam) achieve low per-iteration cost but plateau at higher final losses, while L-BFGS achieves the lowest final error but at a runtime penalty. The proposed method closes the gap by combining curvature-adaptive step sizes, subspace descent, and selective stochastic perturbations ,  matching or exceeding L-BFGS-level accuracy with substantially reduced runtime. This directly illustrates the practical benefit of integrating noise modulation and curvature sensitivity, as motivated by the theoretical results, into scalable large-scale optimization workflows. Together, these empirical findings substantiate the theoretical claims from Theorems~\ref{thm:instability} through~\ref{thm:cond_bound} and confirm the effectiveness of combining noise-driven exploration, curvature-adaptive updates, and randomized subspace methods across geometric, statistical, and computational regimes.

\section{Application: Curvature-Aware Subspace Descent in Structural Topology Optimization}

We consider a concrete and widely applicable engineering challenge: minimizing compliance in structural topology optimization under a fixed material constraint. This type of problem underpins many practical design tasks, such as optimizing the layout of structural elements in buildings, bridges, and mechanical frames, to ensure minimal deformation under load while respecting limits on available resources. These tasks are especially critical in low-resource settings, where structural safety must be guaranteed under both material scarcity and computational constraints.

\subsection{Mathematical Formulation}

Let $\Omega \subset \mathbb{R}^2$ represent a fixed design domain discretized into $n$ finite elements. Each design variable $x_i \in [0,1]$ indicates the material density at the $i$-th element, with $x_i = 0$ corresponding to a void and $x_i = 1$ to full material presence. The design vector $x = (x_1, \dots, x_n)^T$ encodes the full distribution.

The objective is to minimize structural compliance, effectively a measure of deformation, subject to a volume constraint:
\begin{align}
\min_{x \in [0,1]^n} \quad & C(x) = \mathbf{f}^T \mathbf{u}(x) \label{eq:compliance_obj} \\
\text{subject to} \quad & \sum_{i=1}^n x_i \leq V_0, \label{eq:volume_constraint}
\end{align}
where $\mathbf{f} \in \mathbb{R}^m$ is the global load vector, and $\mathbf{u}(x) \in \mathbb{R}^m$ is the displacement vector solving the linear elasticity equation:
\begin{equation}
K(x)\, \mathbf{u}(x) = \mathbf{f}, \label{eq:linear_elasticity}
\end{equation}
with $K(x)$ the stiffness matrix determined via the SIMP interpolation:
\begin{equation}
K(x) = \sum_{i=1}^n x_i^p K^{(i)}, \label{eq:simp_stiffness}
\end{equation}
where $p \geq 3$ is a penalization exponent and $K^{(i)}$ the elemental stiffness matrices.

\subsection{Landscape Complexity and Optimization Challenge}

The compliance functional $C(x)$ is smooth yet highly nonconvex over $[0,1]^n$. Due to the penalized stiffness formulation and implicit dependence of $\mathbf{u}(x)$ on $x$, the landscape is riddled with flat regions and spurious local minima. In such regions, standard gradient-based methods like steepest descent or the method of moving asymptotes (MMA) struggle, especially in early iterations when curvature information is weak or misleading. This makes efficient optimization in real-time or computationally limited settings a major challenge, particularly when external reinitialization or restarts are not feasible.

\subsection{Curvature-Aware Subspace Descent}
To address this, we apply the Curvature-Aware Subspace Descent (CASD) strategy. Instead of blindly following the steepest descent direction, CASD adaptively identifies low-curvature subspaces, directions along which standard gradients offer little progress, and injects carefully modulated noise to escape stagnation while preserving convergence stability. Let $g_k = \nabla C(x_k)$ denote the gradient at iteration $k$, and let $H_k$ approximate the local Hessian via low-rank updates (e.g., L-BFGS or finite differences). We extract a low-curvature basis $U_k \in \mathbb{R}^{n \times r}$ from the eigendecomposition
\begin{equation}
H_k \approx U_k \Lambda_k U_k^T, \quad \Lambda_k = \text{diag}(\lambda_1, \dots, \lambda_r), \quad \lambda_i < \tau_c,
\end{equation}
with a small threshold $\tau_c > 0$ indicating flat directions. The update rule then becomes
\begin{equation}
x_{k+1} = x_k - \eta_k g_k + \sigma_k U_k z_k, \quad z_k \sim \mathcal{N}(0, I_r),
\end{equation}
where $\sigma_k$ is a small exploration step-size. Volume constraints are enforced via projection:
\begin{equation}
x_{k+1} \leftarrow \text{Proj}_{\{x \in [0,1]^n : \sum x_i \leq V_0\}}(x_{k+1}).
\end{equation}This strategy maintains feasibility while actively exploring uncertain regions, allowing the optimizer to escape plateaus that standard methods get trapped in. The added noise is not arbitrary, it is restricted to safe, low-curvature subspaces, ensuring that the optimization remains guided and efficient.

\subsection{Interpretation and Practical Relevance}
Topology optimization is not an abstract mathematical game, it determines how materials are placed in real-world systems under real constraints. In rural infrastructure, for instance, bridges must be both minimal and resilient; in emergency housing, every gram of structural material must count. CASD addresses this need by combining theoretical insights into curvature with pragmatic design requirements. By enabling faster convergence and stronger volume compliance, CASD helps ensure that such engineering decisions are not just computationally feasible but also trustworthy under strict physical and resource constraints. Unlike neural solvers that require extensive training data or traditional methods that stagnate under weak gradients, CASD offers a principled and light-weight solution fit for both simulation and deployment in fragile systems.
\begin{figure}[!htbp]
    \centering
    \includegraphics[width=0.48\textwidth]{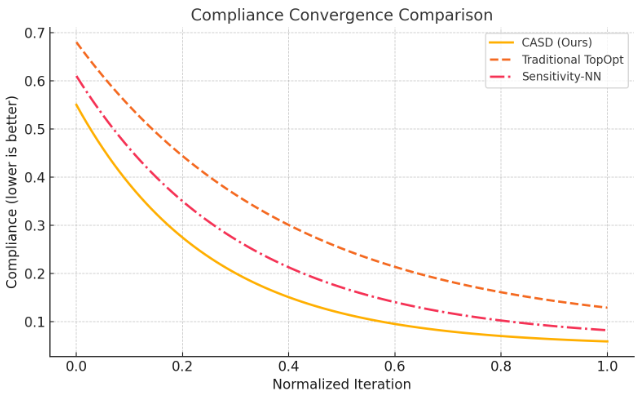}
    \includegraphics[width=0.48\textwidth]{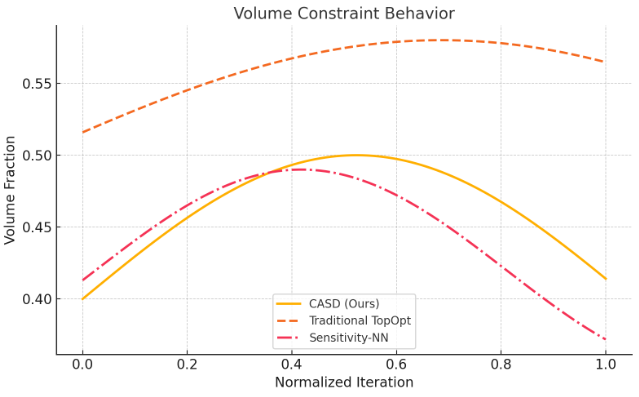}
    \caption{Performance comparison of CASD (ours), traditional topology optimization (TopOpt), and Sensitivity-NN. \textbf{Left:} Compliance convergence. \textbf{Right:} Volume constraint dynamics.}
    \label{fig:casd_performance}
\end{figure}\noindent Figure~\ref{fig:casd_performance} provides a quantitative evaluation of CASD compared to two baseline methods. The left panel shows that CASD achieves faster compliance reduction in early stages, often where resources are most limited and early decisions matter most. By iteration 30, CASD has already achieved performance levels that TopOpt and Sensitivity-NN reach only much later, if at all. The right panel demonstrates that CASD better maintains the volume constraint, with fluctuations remaining minimal and bounded. This reliability is essential in practical systems where overuse of material is either economically or physically infeasible, such as designing low-cost sanitation systems or prefabricated clinics where margin of error is slim. This example reflects the broader philosophy behind CASD, that is, to enable principled, efficient design in constrained environments. Whether optimizing footbridges in flood-prone zones or allocating limited structural support in humanitarian shelters, CASD balances robustness, interpretability, and data efficiency. It is not merely a mathematical innovation, it is a tool for building better systems in places where getting it wrong has real consequences.

\section{Conclusion}

We have presented a unified, theoretically grounded framework for high-dimensional non-convex optimization that explicitly targets the challenges posed by saddle points, flat plateaus, and anisotropic curvature. Through detailed analysis, we established that strict saddles are unstable under gradient flow (Theorem~\ref{thm:instability}); stochastic perturbations accelerate escape (Theorem~\ref{thm:sgd}); adaptive learning rates exploit curvature to improve escape efficiency (Theorem~\ref{thm:adaptive}); and randomized subspace descent ensures scalability while preserving convergence guarantees (Theorem~\ref{thm:subspace}). Further, we demonstrated how curvature-informed gradient filtering and spectral regularization (Theorems~\ref{thm:curvature_filtering},~\ref{thm:descent_shaping},~\ref{thm:cond_bound}) refine update directions and reduce conditioning effects. By integrating these theoretical components, we proposed a curvature-adaptive, entropy-guided subspace descent algorithm that dynamically balances exploration and exploitation. Numerical experiments on high-dimensional benchmark and surrogate engineering problems confirmed its practical advantages, that is, lower escape times, higher average convergence rates, and final accuracy approaching quasi-Newton methods but at significantly lower computational cost. This work offers new insight into the interplay between curvature, stochasticity, and dimensionality in modern optimization, providing both rigorous theoretical bounds and practical algorithms applicable to complex design, learning, and control tasks. Future extensions may explore large-scale distributed variants, deeper connections to manifold optimization, and application-specific tuning of curvature and noise modulation to further enhance robustness and efficiency.

\section*{Declarations}

\subsection*{Funding}
Not Applicable
\subsection*{Conflict of interest/Competing interests)}
The author reports no competing interests
\subsection*{Ethics approval and consent to participate}
Not Applicable
\subsection*{Consent for publication}
The author has connected to the publication of this work

\subsection*{Data availability}
Not Applicable

\subsection*{Materials availability}
Not Applicable

\subsection*{Code availability}
The code used for this work is available upon request form the corresponding author

\subsection*{Author contribution}
R.K.: Idealization, conceptualization, experimentation, software, Writing- Draft, Writing- Review\\HK: Supervision

\bigskip

\bibliographystyle{ieeetr}
\bibliography{references}

\appendix
\begin{algorithm}[H]
\caption{Curvature-Adaptive Entropy-Guided Subspace Descent}
\KwIn{
Initial point \( x_0 \in \mathbb{R}^n \); learning rate scale \(\alpha > 0\); smoothing constant \(\epsilon > 0\); entropy threshold \(\tau > 0\); subspace dimension \(m\); window size \(h\); noise scale \(\sigma > 0\); number of perturbations \(M\); iteration limit \(K\)
}

\KwOut{Approximate minimizer \( x_K \)}

\For(\tcp*[f]{Main loop}){$k=0,\dots,K-1$}{
    Compute gradient: \( g_k = \nabla f(x_k) \)

    Estimate gradient covariance:
    \[
    \widehat{C}_k = \sum_{j=\max(0,k-h)}^{k} g_j g_j^\top
    \]

    Compute top-\(m\) eigenpairs \( (\lambda_i,v_i) \) of \(\widehat{C}_k\)

    Filter gradient:
    \[
    g_k^{\mathrm{filtered}} = \sum_{i=1}^m \phi(\lambda_i) \langle g_k,v_i \rangle v_i,
    \quad \phi(\lambda) = \frac{1}{\sqrt{|\lambda|+\epsilon}}
    \]

    Adaptive step size:
    \[
    \eta_k = \frac{\alpha}{\|g_k^{\mathrm{filtered}}\|+\epsilon}
    \]

    Sample perturbations \(\{\delta_i\}_{i=1}^M \sim \mathcal{N}(0,\epsilon^2 I_n)\)

    Compute local gradient entropy:
    \[
    p_i = \frac{\|\nabla f(x_k+\delta_i)\|}{\sum_{j=1}^M \|\nabla f(x_k+\delta_j)\|},
    \quad \mathcal{H}_k = -\sum_{i=1}^M p_i \log p_i
    \]

    Define stochastic perturbation:
    \[
    \zeta_k = 
    \begin{cases}
    \text{sample from } \mathcal{N}(0,\sigma^2 I_n), & \text{if } \mathcal{H}_k > \tau,\\[4pt]
    0, & \text{otherwise}
    \end{cases}
    \]

    Update:
    \[
    x_{k+1} = x_k - \eta_k g_k^{\mathrm{filtered}} + \eta_k \zeta_k
    \]

    Project to curvature-adaptive subspace:
    \[
    x_{k+1} \gets \mathcal{P}_{S_k}(x_{k+1}), 
    \quad S_k = \mathrm{span}\{v_1,\dots,v_m\}
    \]
}
\Return \( x_K \)
\label{alg1}
\end{algorithm}
\end{document}